\documentclass[10pt]{article}
\usepackage{amssymb}
\usepackage{amsmath}
\usepackage{amsfonts}
\usepackage{amsthm}

\usepackage[english]{babel}
\usepackage[affil-it]{authblk}

\newtheorem{theorem}{Theorem}
\newtheorem{lemma}{Lemma}
\newtheorem{proposition}{Proposition}

\newtheorem{construction}{Construction}
\title{Equitable 2-partitions of Johnson graphs with the second eigenvalue \thanks{The reported study was funded by RFBR according to the research project N 18-31-00126 }}

\author{Konstantin Vorob'ev %
  \thanks{E-mail address: \texttt{vorobev@math.nsc.ru}}}
\affil{Sobolev Institute of Mathematics, Novosibirsk, Russia;\\ Novosibirsk State University, Novosibirsk, Russia}
\begin{document}

\maketitle

\begin{abstract}
We study equitable $2$-partitions of the Johnson graphs
$J(n,w)$ with a quotient matrix containing the eigenvalue $\lambda_2(w,n)=(w-2)(n-w-2)-2$ in its spectrum.   
For any $w\geq 4$ and $n\geq 2w$, we find all admissible quotient matrices of such partitions, and characterize  all these partitions for $w\geq 4$, $n>2w$, and for $w\geq 7$, $n=2w$, up to equivalence.  
\end{abstract}

\section{Introduction}

An $r$-partition $(C_1,C_2,\dots, C_r)$ of the vertex set of a graph is called {\it equitable} with a {\it quotient matrix} $S=(s_{ij})_{i,j\in \{1,2,\dots, r\}}$ if every vertex from $C_i$ has exactly $s_{ij}$ neighbours in $C_j$. The sets $C_1, C_2,\dots, C_r$ are called {\it cells} of the partition. Equitable partitions are also known as perfect colorings, regular partition and partition designs.

A subset of a vertex set of a graph is called a {\it completely regular code} if the distance partition from the subset is equitable. Clearly, any cell of an equitable 2-partition is a completely regular code. 

It is known \cite{Cvetkovic} that an eigenvalue of a quotient matrix of an equitable partition of a graph must be an eigenvalue of the adjacency matrix of this graph. In this paper, by an eigenvalue of a partition we will understand an eigenvalue of its quotient matrix.

The vertices of the Johnson graph $J(n,w)$ are the binary vectors of length $n$ with $w$ ones, where two
vectors are adjacent if they have exactly $w-1$ common ones. This graph is distance-regular (see, for example, \cite{BCN}) with $w+1$ distinct eigenvalues $\lambda_i(n,w)=(w-i)(n-w-i)-i$, $i=0,1, \dots w$.  

In this work we consider equitable $2$-partitions of a Johnson graph $J(n,w)$, $w\geq 3$ with a quotient matrix having eigenvalue $\lambda_2(n,w)$ (another eigenvalue is a degree of the graph $w(n-w)$). The problem of existence of equitable $2$-partition of Johnson graphs with given quotient matrix is far from solving. In particular, it includes a famous Delsarte conjecture about non-existence of $1$-perfect codes in the Johnson scheme (see, for example, \cite{AM2011SEMR}). 

Equitable $2$-partitions were studied by Avgustinovich and Mogilnykh in several papers \cite{AM2008,AM2011JAIM,AM2011SEMR,Mog2007,Mog2009}.

One of possible ways to solve the problem of existence is to characterize partitions with certain eigenvalues. Equitable $2$-partitions of the graph $J(n,w)$ with the eigenvalue $\lambda_1(n,w)$ were characterized by Meyerowitz \cite{Meyer}. In \cite{GG2013} Gavrilyuk and Goryainov found all realizable quotient matrices (i.e. quotient matrices of some existing partitions) of equitable $2$-partitions of $J(n,3)$ with second eigenvalue $\lambda_2(n,3)$ for odd $n$ and announced the solution for even $n$.

In this paper we study equitable $2$-partitions of Johnson graphs $J(n,w)$, $w\geq 4$.
The paper is organized as follows. In Section $2$, we introduce all necessary definition and basic statements. Section $3$ is devoted to $\lambda_1(n,w)$-eigenfunctions of $J(n,w)$ taking not more than $3$ distinct values and their properties - the main tool in this paper. In Section $4$, we prove that there are no equitable $2$-partitions of $J(n,w)$ with second eigenvalue for $w\geq 4$ and $n>2w$. In Section $5$ we find all realizable quotient matrices of such partitions for $n=2w$, $w\geq 4$, and obtain a full characterization for $w\geq 7$. In particular, we find $2$ new infinite series of partitions for $n=2w$.

\section{Preliminaries}

Let $G=(V,E)$ be a graph. A real--valued function $f:V\longrightarrow{\mathbb{R}}$ is called a {\em $\lambda$--eigenfunction} of $G$ if the equality $$\lambda\cdot f(x)=\sum_{y\in{(x,y)\in E}}f(y)$$ holds for any $x\in V$ and $f$ is not the all-zero function. Note that the vector of values of a $\lambda$--eigenfunction is an eigenvector of the adjacency matrix of $G$ with an eigenvalue $\lambda$. The support of a real--valued function $f$ is the set of nonzeros of $f$. The cardinality of the support of $f$ is denoted by $|f|$.

Given a real-valued $\lambda_i(n,w)$-eigenfunction $f$ of
$J(n,w)$ for some $i\in \{0,1,\dots,w\}$ and $j_1,j_2\in
\{1,2,\dots,n\}$, $j_1<j_2$, define a {\it a partial difference of $f$} -- a real-valued function
$f_{j_1,j_2}$ as follows: for any vertex $y=(y_1,y_2, \dots
,y_{j_1-1},y_{j_1+1}, \dots ,y_{j_2-1},y_{j_2+1},\dots,y_n)$ of
$J(n-2,w-1)$
\begin{flushright}$f_{j_1,j_2}(y)=f(y_1,y_2, \dots
,y_{j_1-1},1,y_{j_1+1}, \dots ,y_{j_2-1},0,y_{j_2+1},\dots,y_n)\,$
\end{flushright} \begin{flushright}
 $-f(y_1,y_2, \dots ,y_{j_1-1},0,y_{j_1+1}, \dots ,y_{j_2-1},1,y_{j_2+1}, \dots ,y_n).$
\end{flushright}

\begin{lemma}\label{L:part_diff}(\cite{VMV})
 If f is a $\lambda_i(n,w)$-eigenfunction of
$J(n,w)$ then $f_{j_1,j_2}$ is a
$\lambda_{i-1}(n-2,w-1)$-eigenfunction of $J(n-2,w-1)$ or the all-zero function.
\end{lemma}

As we see, given an eigenfunction $f$ from Lemma \ref{L:part_diff} we
obtain the eigenfunctions $f_{j_1,j_2}$ in the Johnson
graph with smaller parameters for every distinct coordinates 
$j_1,j_2$. Note, that in some cases the resulting function $f_{j_1,j_2}$ is
just the all-zero function. Moreover, the set of all-zero partial differences induces a partition on the set of coordinates positions.

\begin{lemma}\label{L:zero_part_diff}(\cite{VMV})
 Let $f\in J(n,w)\rightarrow \mathbb{R}$. Let $f_{i_1,i_2}\equiv 0$ and $f_{i_1,i_3}\equiv 0$ for some pairwise distinct $i_1,i_2,i_3\in \{1,2,\dots n\}$. Then $f_{i_2,i_3}\equiv 0$.
\end{lemma}

We will say that functions $f_1,f_2:J(n,w)\rightarrow \mathbb{R}$ are equivalent if there exist a permutation $\pi\in S_n$ such that $\forall x\in J(n,w)$ we have $f_1(x)=f_2(\pi x)$. Two equitable $2$-partitions $(C_1,C_2)$ and $(C_1',C_2')$ of the graph $J(n,w)$ are equivalent if the characteristic function $\chi_{C_1}$ is equivalent to $\chi_{C_1'}$ or $\chi_{C_2'}$. 

Let us discuss some basic properties of equitable $2$-partitions of $J(n,w)$. Such a partition has a quotient matrix $[[a,b][c,d]]$. Since a Johnson graph is regular of degree $w(n-w)$, we have $a+b=c+d=w(n-w)$, where $a,b,c,d$ are non-negative integers. Since $J(n,w)$ is connected, $b>0$ and $c>0$. Without loss of generality, we always consider the case $b\geq c$. It easy to prove, that $a-c$ is an eigenvalue of the quotient matrix. Therefore, for $a-c=\lambda_2(n,w)$ we have

\begin{proposition}\label{P:matrix}
Let $(C_1,C_2)$ be an equitable $2$-partition of $J(n,w)$ with second eigenvalue. Then the partition has the quotient matrix
$[[w(n-w)-b,b][2n-2-b,w(n-w)-2n+2+b]]$ for some $b \in \{n-1,n,\dots, 2n-1\}$.    
\end{proposition}
We will also need the following useful well-known property of equitable $2$-partitions for $n=2w$. 

\begin{lemma}\label{L:antipodality}
Let $(C_1,C_2)$ be an equitable $2$-partition of $J(2w,w)$ with second eigenvalue. Let $x\in C_1 (C_2)$. Take the vertex $x'\in J(2w,w)$ such that $x$ and $x'$ have distinct values in all $2w$ coordinate positions. Then $x'\in C_1 (C_2)$.   
\end{lemma}
In the following Section, we will be focused on eigenfunctions taking a few number of values. We are going to find and prove some structural properties of such functions that will help us to characterize equitable $2$-partitions later.

\section{Eigenfunctions taking three values}

Consider a characteristic function of one cell of some equitable $2$-partition of $J(n,w)$ with the eigenvalue $\lambda_2(n,w)$ and take some partial difference of this function. By Lemma \ref{L:part_diff} the resulting function is a $\lambda_1(n-2,w-1)$-eigenfunction of $J(n-2,w-1)$ or the all-zero function. In any case, this partial difference may take only three distinct values $-1$, $0$, $1$. As we see, the problem of constructing equitable $2$-partition with $\lambda_2(n,w)$ may be reduced to the problem of constructing $\lambda_1(n-2,w-1)$-eigenfunctions with some restrictions. The following theorem gives a full classification of $\lambda_1(n-2,w-1)$-eigenfunctions we are interested in.   

\begin{theorem}\label{T:Derivates}
 If $f:J(n,w)\rightarrow \{-1,0,1\}$ is a $\lambda_1(n,w)$-eigenfunction of
$J(n,w)$, $f\not\equiv 0$, $w\geq 2$, then $f$ is equivalent up to multiplication by a non-zero constant to one of the following functions:
\begin{enumerate}
 \item $f_1(x)= \begin{cases}
1, x_1=1,\,x_2=0\\
-1, x_1=0,\,x_2=1\\
0,&\text{otherwise.} \end{cases}$, $x=(x_1,x_2,\dots,x_n)\in J(n,w)$, $w\geq 2$ and $n\geq 2w$

 \item $f_2(x)= \begin{cases}
1, x_1=1,\,x_2=1\\
-1, x_1=0,\,x_2=0\\
0,&\text{otherwise.}  \end{cases}$, $x=(x_1,x_2,\dots,x_n)\in J(n,w)$, $w\geq 2$ and $n=2w$.

 \item $f_3(x)= \begin{cases}
1, x_1=1,\\
-1, x_1=0,\\
0,&\text{otherwise.}  \end{cases}$, $x=(x_1,x_2,\dots,x_n)\in J(n,w)$, $w\geq 2$ and $n=2w$.

  \item $f_4(x)= \begin{cases}
1, Supp(x)\subseteq \{1,2,\dots,\frac{n}{2}\},\\
-1,  Supp(x)\subseteq \{\frac{n}{2}+1,\frac{n}{2}+2,\dots, n\},\\
 0,&\text{otherwise.} \end{cases}$, $x=(x_1,x_2,\dots,x_n)\in J(n,w)$, $w=2$, $n \geq 2w$ and $n$ is even.

\end{enumerate}
\end{theorem}

\begin{proof}
Let us consider the function $g:J(n,1)\rightarrow \mathbb{R}$ such that $g(x)=\sum_{y\in J(n,w)|y\geq x}{f(y)}$, where by $x\geq y$ we understand that set of coordinate positions of $x$ with ones is a subset of the corresponding set for $y$. The function $g$ is a so-called induced function. 

For the Johnson graph, it is known that if $f$ is a $\lambda_1(n,w)$-eigenfunction of $J(n,w)$ then $\lambda_1(n,1)$-eigenfunction of $J(n,1)$. The set of vertices of $J(n,1)$ is exactly the set of unit vectors $e_i$, $i=1,2,\dots, n$ with one in the coordinate $i$. Let us denote by $a_i$ the value $g(e_i)$. Without loss of generality one may assume that $a_1\geq a_2\geq \dots a_n$. Since $g$ is a $\lambda_1(n,1)$-eigenfunction of $J(n,1)$, the function must be orthogonal to a constant function. Therefore, $\sum_{i}{a_i}=0$. 

It is also known (see, for example \cite{Delsarte}), that for $x=(x_1,x_2,\dots, x_n)\in J(n,w)$, we have $$f(x)=\alpha\sum_{i| x_i=1}{a_i},$$ where $\alpha$ is a non-zero constant.
By the theorem hypothesis, $f$ takes exactly 3 different values, it gives some constraints on the multiset $A=\{a_i | i\in \{1, 2,\dots, n\}\}$. The rest of the proof is based on the analysis of this multiset.

Suppose that there are at least $4$ pairwise distinct elements in $A$, say $b_1>b_2>b_3>b_4$. Let $s_1$, $s_2$, $s_3$ and $s_4$ be coordinate positions such that $g(s_i)=b_i$, $i=1,2,3,4$.  
For $w=2$, we have at least $4$ different values of $f$: $f(e_{s_1}+e_{s_2})$, $f(e_{s_1}+e_{s_3})$, $f(e_{s_1}+e_{s_4})$ and $f(e_{s_3}+e_{s_4})$. In case $w\geq 3$, we take any $y\in J(n,w-1)$ having zeros at positions $s_1$,$s_2$,$s_3$ and $s_4$; it is clear that $|\{f(y+e_{s_i}) | i=1,2,3,4\}|=4$ and we get a contradiction.     
 
Suppose that there are exactly $3$ distinct elements $b_1,b_2,b_3\in A$, $b_1>b_2>b_3$, and $s_1$, $s_2$, $s_3$ be corresponding coordinate positions. If $w=2$ then $f(e_{s_1}+e_{s_2})$, $f(e_{s_1}+e_{s_3})$, $f(e_{s_2}+e_{s_3})$ are distinct values. It is easy to see, that $f(e_{s_1}+e_{s_2})> f(e_{s_1}+e_{s_3})>f(e_{s_2}+e_{s_3})$. Therefore, $f(e_{s_1}+e_{s_2})=-f(e_{s_2}+e_{s_3})=1$ and $f(e_{s_1}+e_{s_3})=0$. So we conclude that $b_1=-b_3$ and $b_2=0$. If the multiset $A$ contains one more element $c$ (with corresponding coordinate $s_4$) which is equal to $b_1$ or $b_3$ then $f$ takes one more value except $\{-1,0,1\}$ ($f(e_{s_1}+e_{s_4})>1$ or $f(e_{s_3}+e_{s_4})<-1$ respectively).

 Consequently, $a_1=b_1$, $a_n=b_3$, $a_2=a_3=\dots =a_{n-1}=0$ and $f$ is equivalent to $f_1$ for $w=2$ from the statement of the theorem. The next case is $w\geq 3$. Suppose that the element $b_1$ has a multiplicity more then $1$ in $A$, so $a_{s_1}=b_1$,  $a_{s_2}=b_2$,  $a_{s_3}=b_3$,  $a_{s_4}=b_1$ for some pairwise distinct integers $s_1,s_2,s_3,s_4$. Take some vector $\bar{y}$ from $J(n,w-2)$ with zeros in coordinate positions corresponding to $b_i$, $i=1,2,3,4$. It is easy to check that $|\{f(y+e_{s_1}+e_{s_4}),f(y+e_{s_1}+e_{s_2}),f(y+e_{s_1}+e_{s_3}),f(y+e_{s_2}+e_{s_3}) \}|=4$, and it contradicts to the fact that $f$ takes only $3$ distinct values.  
 
 Providing similar arguments one can show that a multiplicity of $b_3$ in $A$ also equals $1$. So we conclude, that multiplicities of $b_1$,$b_2$ and $b_3$ in $A$ are $1$, $n-2$ and $1$ respectively. In particular, consider the following $4$ values of $f$: $\alpha(b_1+(w-1)b_2)$, $\alpha(b_3+(w-1)b_2)$,
 $\alpha(wb_2)$, $\alpha(b_3+(w-2)b_2+b_1)$. Clearly, the first three of them are pairwise distinct and the fourth can not be equal to the first and the third. The only possible case is that $wb_2=b_3+(w-2)b_2+b_1$, so $b_1=-b_3$ and $b_2=0$. It means, that $f$ is equivalent to $f_1$ for $w\geq 3$.
 
The last case is that there are exactly $2$ distinct elements in $A$. Let $A$ contain $k_1$ elements $a_1$ and $k_2$ elements $a_2$, $k_1+k_2=n$. By orthogonality to a constant function we know $k_1a_1+k_2a_2=0$ and $a_1>0$, $a_2<0$. Without loss of generality we may consider the case $k_1<k_2$ (otherwise, we provide our arguments for a function $-f$).

Let $w$ be equal to $2$. Then $\frac{1}{\alpha}f$ takes $3$ distinct values: $2a_1$, $(a_1+a_2)$ and $2a_2$, and one of them must be $0$. Obviously, $a_1=-a_2$, $k_1=k_2=\frac{n}{2}$ and we found a function $f$ which is equal to $f_4$ from the theorems statement. 

So, in the rest of the proof we have $w\geq 3$. If $k_1\geq 3$ the function $\frac{1}{\alpha}f$ takes at least $4$ distinct values: $wa_2$, $(a_1+(w-1)a_2)$, $(2a_1+(w-2)a_2)$, $(3a_1+(w-3)a_2)$. It means, that we have only to possible cases: $k_1=1$ and $k_2=2$. In the first case, $\frac{1}{\alpha}f$ takes exactly two distinct values:  $wa_2$, $(a_1+(w-1)a_2)$. These values can not be equal to $0$, so they are $-1$ and $1$ respectively. It gives us $a_1=-(2w-1)a_2$ and $n=2w$ by the constant function. As one can see, $f$ is equal to $f_3$. In the second case,  $\frac{1}{\alpha}f$ takes exactly three distinct values:  $wa_2$, $(a_1+(w-1)a_2)$, $(2a_1+(w-2)a_2)$. Clearly, these values must be equal to $-1$,$0$ and $1$ respectively, which immediately gives us $a_1=-(w-1)a_2$ and $n=2w$ by orthogonality to the constant function. In this case, we build the function $f$ which is equal to $f_2$ and finish the proof.

\end{proof}

In the next Section, based on this theorem we prove that there are no equitable $2$-partitions of $J(n,w)$ with second eigenvalue for $n>2w$.

\section{Johnson graphs $J(n,2w)$, $n>2w$}  

Surprisingly, there are no equitable $2$-partitions in Johnson graphs for $n>2w$, $w>3$ with the second eigenvalue. 
\begin{theorem}\label{T:n>2w}
 There are no equitable $2$-partitions in a Johnson graph $J(n,w)$, $n>2w$, $w>3$, with the quotient matrix $[[w(n-w)-b,b],[2n-2-b,w(n-w)-2n+2+b]]$, $b \in \{n-1,n,\dots, 2n-1\}$. 
 
 \end{theorem}
 \begin{proof}
   Suppose that $(C_1,C_2)$ is an equitable $2$-partition with the quotient matrix $[[w(n-w)-b,b],[2n-2-b,w(n-w)-2n+2+b]]$ for some $b \in \{n-1,n,\dots, 2n-1\}$. It easy to check, that the function $f:J(n,w)\rightarrow \mathbb{R}$ defined as $b\chi_{C_1}-c\chi_{C_2}$ is a $\lambda_2(n,w)$-eigenfunction of $J(n,w)$ and takes exactly two values: $b$ and $-c$, where $c=2n-2-b$. Consider a function $g=\frac{f_{i_1,i_2}}{b+c}$ defined on vertices of $J(n-2,w-1)$ for some $i_1$, $i_2$, $1\leq i_1<i_2\leq n$. Clearly, $g:J(n-2,w-1)\rightarrow \{-1,0,1\}$ and $g$ is a $\lambda_1(n-2,w-1)$-eigenfunction of $J(n-2,w-1)$. Since $|C_1|>0$ and $|C_2|>0$ there is a pair $(i_1,i_2)$ that the function $g$ is not the all-zero function. Without loss of generality we make take $i_1=1$ and $i_2=2$.
By Theorem \ref{T:Derivates} we have $$g(\bar{x})= \begin{cases}
1, x_3=1,\,x_4=0\\
-1, x_3=0,\,x_4=1\\
0,&\text{otherwise,} \end{cases}$$ $\bar{x}=(x_3,x_4,\dots,x_n)\in J(n-2,w-1).$ 
Therefore, we have the following equalities
 $$f(1010\bar{z})=f(0101\bar{z})=b,\, \bar{z}\in J(n-4,w-2),$$
 $$f(1001\bar{z})=f(0110\bar{z})=-c,\, \bar{z}\in J(n-4,w-2),$$
 $$f(1000\bar{z})=f(0100\bar{z}),  \bar{z}\in J(n-4,w-1), $$
 $$f(1011\bar{z})=f(0111\bar{z}),  \bar{z}\in J(n-4,w-3). $$
It follows from these equalities that $f_{3,4}(10\bar{z})=b+c$ and $f_{3,4}(01\bar{z})=-(b+c)$ for all $\bar{z}\in J(n-4,w-2)$. By Theorem \ref{T:Derivates} $f_{3,4}$ is also equivalent up to a multiplication by a scalar to $g$ and it gives us the following equalities:
  $$f(0010\bar{z})=f(0001\bar{z}),  \bar{z}\in J(n-4,w-1), $$
  $$f(1110\bar{z})=f(1101\bar{z}),  \bar{z}\in J(n-4,w-3). $$

Our next goal is to show that $f_{i_1,i_2}\equiv 0$ for $i_1\neq i_2$, $i_1,i_2\in \{5,6,\dots, n\}$.
Suppose that $f_{i_1,i_2}\not\equiv 0$ for some $i_1$, $i_2$. Without loss of generality one may take $i_1=5$, $i_2=6$. By Theorem \ref{T:Derivates} and by similar arguments we provided for the function $g$ there are $i_3,i_4\in \{1,2,\dots n\}\setminus \{5,6\}$ such that $f(\bar{z})=b$ for $z\in A_1$ and $f(\bar{z})=-c$ for $z\in A_2$, where $$A_1=\{z\in J(n,w)|z_{i_1}+z_{i_2}=z_{i_3}+z_{i_4}=1 ,\, z_{i_1}+z_{i_3}=0\,\,\, \text{or}\,\,\, 2    \},$$ $$A_2= \{z\in J(n,w) |z_{i_1}+z_{i_2}=z_{i_3}+z_{i_4}=1 ,\, z_{i_1}+z_{i_3}=1\}.$$ 
As we defined above, $i_1,i_2\not\in \{1,2,3,4\}$. 

Let us consider numbers $i_3$ and $i_4$.  
If $i_3,i_4\not\in \{1,2,3,4\}$ then we can take any $z'\in A_1$ such that $(z_1,z_2,z_3,z_4)=(0,1,1,0)$.  It is easy to see that $f(z')=b$, but we know that $f(z')=-c$, so we get a contradiction.    

The next case is $i_4\not\in \{1,2,3,4\}$, $i_3\in  \{1,2,3,4\}$. Without loss of generality we take $i_3\in \{1,2\}$. Now we choose $z\in A_1$ such that $z_1+z_2=1$, $z_{3}=1$, $z_{4}=0$ and a vertex $z'$ obtained from $z$ by changing coordinate positions $3$ and $4$. Clearly, $z'\in A_1$, so $f(z')=b$. However, $f(z')\neq f(z)$, so this case is not possible. 

The last case is  $i_3,i_4 \in \{1,2,3,4\}$. Let $i_3,i_4\in \{1,2\}$ (the case $i_3,i_4\in \{3,4\}$ may be considered providing similar arguments). Then we take some $z\in A_1$ such that $z_3+z_4=1$. Again we build $z'$ by changing the third and the fourth coordinate positions and it still must be an element of $A_1$, so $f(z')=b$, so we have a contradiction. If $i_3\in \{1,2\}$, $i_4\in \{3,4\}$ we take some $z\in A_1$ such that $z_1+z_2=1$ and $z_3+z_4=1$. We build $z'$ by changing coordinate positions $i_1$ and $i_4$. It is known that $z' \in A_2$, so $f(z')=-c$ but we know that $f(z)=f(z')$. 

Finally, we proved that $f_{i_1,i_2}\equiv 0$ for $i_1\neq i_2$, $i_1,i_2\in \{5,6,\dots, n\}$. In other words, the value $f(z)$ depends only on $z_1$, $z_2$, $z_3$ and $z_4$ and does not depend on the distribution of ones in $\{5,6,\dots, n\}$.     
 
Note that in the arguments above we always took vectors $z$ and $z'$ containing not more than $4$ ones, so the condition $w>3$ guarantees us correctness of these steps. 

Let us consider the vertex $\bar{x}=(0,0,0,0,\bar{x'})$ for some $\bar{x'}\in J(n-4,w)$. By the arguments provided for all vertices $\bar{y}\in J(n,w)$ having zeros in the first four coordinate positions, we have $f(\bar{x})=f(\bar{y})$. In other words, all these vertices are elements of one cell of partitions. Similarly, all vertices of the form $(a_1,a_2,0,0,\bar{z'})$, $(a_1+a_2)=1$ and $\bar{z'}\in J(n-4,w-1)$, belong to one cell and all vertices of the form $(0,0,a_1,a_2,\bar{z'})$, $(a_1+a_2)=1$ and $\bar{z'}\in J(n-4,w-1)$, also  belong to one cell. Hence, the number o neighbours of $x$ from cell which does not contain $x$ may take only values $2w$ or $4w$. By similar arguments for the vertex $z=(1,1,1,1,\bar{x'})$ for some $\bar{x'}\in J(n-4,w-4)$
we conclude that the vertex $z$ may have only $2(n-w)$ or $4(n-w)$ neighbours from the cell not containing $z$.

 By Proposition \ref{P:matrix} we know that $b+c=2n-2$. Since $4(n-w)>2n-2$ for $n>2w$, we conclude that $b=2(n-w)$ and $b=2w-2$, or $c=2(n-w)$ and $b=2w-2$. 
 
 By our agreement $b\geq c$, so the only possible quotient matrix of our partition is $[[w(n-w)-2(n-w),2(n-w)],[2w-2,w(n-w)-(2w-2)]]$. Let us turn on the vertex $x$ again. If $x\in C_2$ then $2w-2=2w$ or $2w-2=4w$ and we get a contradiction. Hence, $x\in C_1$ and $2(n-w)=2w$ or $2(n-w)=4w$. 
 
 The first case leads us to $n=2w$ but we have $n>2w$, so finally we have $n=3w$ and the quotient matrix $[[2w^2-4w,4w],[2w-2,2w^2-2w+2]]$. Since $x\in C_1$, we have $B=\{y=(y_1,y_2,\dots, y_n)\in J(n,w)| y_1+y_2+y_3+y_4=1\}\subseteq C_2$. Take some vertex from $B$, according to the quotient matrix, this vertex has exactly $2w-2$ neighbours from $C_1$. By simple counting it has $2w-3$ neighbours in $\{y=(y_1,y_2,\dots, y_n)\in J(n,w)| y_1+y_2+y_3+y_4=0\}\cap C_1$ and $w-1$ neighbours in $\{y=(y_1,y_2,\dots, y_n)\in J(n,w)| y_1+y_2=y_3+y_4=1\}\cap C_1$. Consequently, by simplifying $(2w-3)+(w-1)\leq 2w-2$ we obtain $w\leq 2$ and get a contradiction.               

 \end{proof}
 
 The proof of the Theorem \ref{T:n>2w} is based on the fact that any partial difference $f_{i,j}$ of the function $f$ is equal to the all-zero function or to $f_1$ from the Theorem \ref{T:Derivates}. In the case $n=2w$ there are other variants for partial differences $f_{i,j}$ and one requires other approaches to gain the characterization of equitable $2$-partitions with $b+c=4w-2$.

\section{Johnson graphs J(2w,w)}

Before we start working with partial differences for $n=2w$, let us consider some previously known and new constructions of equitable $2$-partitions of $J(2w,w)$.

\begin{construction}\label{C:w}
Let $C=(C_1,C_2)$ be a partition of the set of vertices of $J(2w,w)$, $w\geq 3$, defined by the following rule:
$$C_1=\{(x_1,x_2,x_3,x_4,x_5\dots,x_n)\in J(2w,w)| (x_1,x_2,x_3,x_4,x_5)\in B\},
C_2=J(2w,w)\setminus C_1,$$

where $B= \{(1,0,0,0,0),(1,1,0,0,0),(1,0,1,0,0),(0,0,0,1,1),\\(0,1,1,1,1),(0,0,1,1,1),(0,1,0,1,1),(1,1,1,0,0)\}$.   Then $C=(C_1,C_2)$ is equitable with the quotient matrix $[[w^2-3w+2,3w-2],[w,w^2-w]]$.

\end{construction}

\begin{construction}\label{C:2w_easy}
Let $C=(C_1,C_2)$ be a partition of the set of vertices of $J(2w,w)$, $w\geq 3$, defined by the following rule:
$$C_1=\{(x_1,x_2,\dots,x_n)\in J(2w,w)|x_1+x_2=0 \text{ or } 2\}, C_2=J(2w,w)\setminus C_1.$$ Then $C=(C_1,C_2)$ is equitable with the quotient matrix $[[w^2-2w,2w],[2w-2,w^2-2w+2]]$.

\end{construction}

\begin{construction}\label{C:2w_hard}
Let $C=(C_1,C_2)$ be a partition of the set of vertices of $J(2w,w)$, $w\geq 5$, defined by the following rule:
$$C_1=\{(x_1,x_2,x_3,x_4,x_5,\dots,x_n)\in J(2w,w)| (x_1,x_2,x_3,x_4,x_5) \in B \}, C_2=J(2w,w)\setminus C_1,$$
where $B=\{(0,0,0,0,0),(0,0,1,0,0),(0,0,0,1,0),(0,0,0,0,1), (1,0,1,0,0),\\(0,1,0,1,0),(0,0,1,0,1),(0,0,0,1,1), (1,1,1,1,1),(1,1,0,1,1),(1,1,1,0,1),\\(1,1,1,1,0), (0,1,0,1,1),(1,0,1,0,1),(1,1,0,1,0),(1,1,1,0,0) \}$. Then $C=(C_1,C_2)$ is equitable with the quotient matrix $[[w^2-2w,2w],[2w-2,w^2-2w+2]]$.

\end{construction}

\begin{construction}\label{C:3w}
Let $C=(C_1,C_2)$ be a partition of the set of vertices of $J(2w,w)$, $w\geq 3$, defined by the following rule:
$$C_1=\{(x_1,x_2,x_3,\dots,x_n)\in J(2w,w)| (x_1,x_2,x_3) \in \{(0,0,0),(1,1,1)\} \}, C_2=J(2w,w)\setminus C_1.$$ Then $C=(C_1,C_2)$ is equitable with the quotient matrix $[[w^2-3w,3w],[w-2,w^2-2+2]]$.

\end{construction}

Let us note, that the proof of correctness for the four construction listed above is simple and straightforward. One just need to count the number of neighbours from both cells for vertices of different "types" - starting subvectors of a small length not greater than $5$. Moreover, in fact one need to check only a half of "types" because of Lemma \ref{L:antipodality}.
 
Constructions \ref{C:2w_easy} and \ref{C:3w} were known before. As far as the author knows, Constructions \ref{C:w} and \ref{C:2w_hard} are new.

Now we are going to characterize all equitable $2$-partitions of $J(2w,w)$ with second eigenvalue for $w\geq 7$.  
As we know, for $n=2w$ the quotient matrix of an equitable $2$-partition $(C_1,C_2)$ with second eigenvalue is equal to $[[w^2-b,b],[4w-2-b,w^2-4w+b+2]]$ for some $b\in \{2w-1,2w,\dots 4w-3\}$. 

\begin{proposition}\label{Derivatves:n=2w}
 Let $C=(C_1,C_2)$ be an equitable $2$-partition in a Johnson graph $J(2w,w)$, $w\geq 4$, with the quotient matrix $[[w^2-b,b],[4w-2-b,w^2-4w+2+b]]$, $b\in \{2w-1,2w,\dots 4w-3\}$. Let some partial difference of a function $g=\frac{b\chi_{C_1}-c\chi_{C_2}}{b+c}$, where $c=4w-2-b$, be equivalent to the function $f_3$ from Theorem \ref{T:Derivates}. Then $b=2w$ and $C$ is equivalent to the partition from Construction \ref{C:2w_easy}. 
 \end{proposition}
  
\begin{proof}
Let some partial difference of $g$ be equivalent to the function $f_3(x)$ from Theorem \ref{T:Derivates}. Without loss of generality we may assume that $$g_{1,2}(x_3,x_4,\dots,x_n)=   \begin{cases}
1, x_3=1,\\
-1, x_3=0,\\
0,&\text{otherwise,}  \end{cases}  $$ $x=(x_3,x_4,\dots, x_n)\in J(n-2,w-1)$.

By definition of the function $g$ this equality allows us to reconstruct some values of $f$: $f(1,0,1,\bar{y})=b$ and $f(0,1,1,\bar{y})=-c$ for all $\bar{y}\in J(2w-3,w-2)$, 
$f(1,0,0,\bar{y})=-c$ and $f(0,1,0,\bar{y})=b$ for all $\bar{y}\in J(2w-3,w-1)$.

Now consider sets $A_{i,j}=\{(z_1,z_2,\dots, z_n)\in J(2w,w)|z_1=i,\, z_2=j\}$, $i,j\in\{0,1\}$. It is easy to check that an arbitrary vertex from $A_{1,1}$ has exactly $w$ neighbours in $(A_{1,0}\cup A_{0,1})\cap C_1$ and $w$ neighbours in $(A_{1,0}\cup A_{0,1})\cap C_2$. Therefore the partition $C'=(C_1\cap A_{1,1},C_2\cap A_{1,1})$ of Johnson graph $J(n-2,w-2)$ with the set of vertices $A_{1,1}$ is equitable with the quotient matrix $[[w^2-w-b,b-w],[3w-2-b, w^2-5w+2+b]]$. By direct calculation one may see that this matrix has the eigenvalue $\lambda_1(2w-2,w-2)$. By Meyerowitz classification \cite{Meyer} this partition has simple structure. There is one coordinate position $j\in\{3,4,\dots 2w\}$ such that all vectors having $0$ in $j$-th coordinate are elements of $C_1'$ and all vectors having $1$ are elements of $C_2'$.

So we know the sets $A_{1,1}\cap C_1$ and $A_{1,1}\cap C_2$. By Lemma \ref{L:antipodality} we also know $A_{0,0}\cap C_1$ and $A_{0,0}\cap C_2$.

 As we see, there is one special coordinate position in the Meyerowitz construction. Therefore, we have two different cases.
 \begin{enumerate}

 \item This position coincides with $x_3$. Clearly, we have $C'_1=A_{1,1}\cap C_1=\{(1,1,1,\bar{y})|y\in J(2w-3,w-3)\}$ and $C'_2=A_{1,1}\cap C_2=\{(1,1,0,\bar{y})|y\in J(2w-3,w-2)\}$. By Lemma \ref{L:antipodality} we can find the whole cells: $$C_1=\{(x_1,x_2,x_3,\dots x_n)\in J(2w,w)| (x_1,x_2,x_3)\in \{(1,0,1),(0,1,0),(1,1,1),(0,0,0)\}\},$$ 
  $$C_2=\{(x_1,x_2,x_3,\dots x_n)\in J(2w,w)| (x_1,x_2,x_3)\in \{(1,0,0),(0,1,1),(1,1,0),(0,0,1)\}\}$$

 Now we need to check that the partition with these cells is equitable with a quotient matrix $[[w^2-2w,2w],[2w-2, w^2-2w+2]]$. By our arguments we do not need to do it for vertices with $(x_1,x_2)\in \{(1,1),(0,0)\}$.
 Consider for example some $x=(1,0,1,\bar{x'})\in C_1$. By simple counting we find that $x$ has $w-1$, $1$, $1$ and $w-1$ neighbours in $C_2$ having values in the first three coordinates $(1,0,0)$,$(0,1,1)$,$(1,1,0)$ and $(0,0,1)$ respectively. For remaining $3$ types of vertices in $C_1$ and $C_2$ the counting is similar.  
 We see that $(x_1,x_2,x_3,\dots x_{2w})\in C_1$ if and only if $x_1+x_2$ equal $0$ or $2$, so the partition is equivalent to Construction \ref{C:2w_easy}
 
 \item This position does not coincide with $x_3$. Without loss of generality we take it as $x_4$. Here we have 
  $C'_1=A_{1,1}\cap C_1=\{(1,1,0,1,\bar{y})|y\in J(2w-4,w-3)\}\cup \{(1,1,1,1,\bar{y})|y\in J(2w-4,w-4)\} $ and $C'_2=A_{1,1}\cap C_2=\{(1,1,0,0,\bar{y})|y\in J(2w-4,w-2)\}\cup \{(1,1,1,0,\bar{y})|y\in J(2w-4,w-3)\} $. By Lemma \ref{L:antipodality} we can find the whole cells: 
   $$C_1=\{(x_1,x_2,x_3,x_4,\dots x_n)\in J(2w,w)| (x_1,x_2,x_3,x_4)\in B_1\}, $$ where $B_1=\{(1,0,1,0),(1,0,1,1),(0,1,0,1),(0,1,0,0),(1,1,0,1),(1,1,1,1),(0,0,1,0),(0,0,0,0)\},$
    $$C_2=\{(x_1,x_2,x_3,x_4,\dots x_n)\in J(2w,w)| (x_1,x_2,x_3,x_4)\in B_2\}, $$ where $B_2=\{(1,0,0,0),(1,0,0,1),(0,1,1,1),(0,1,1,0),(1,1,0,0),(1,1,1,0),(0,0,1,1),(0,0,0,1)\}.$
  
  The verification that $C=(C_1,C_2)$ is equitable with the quotient matrix $[[w^2-2w,2w],[2w-2, w^2-2w+2]]$ is also simple and straightforward (again we do need to check it for vectors having two zeros or two ones in the first two coordinate positions).
  
  Again we see that $(x_1,x_2,x_3,x_4,\dots x_{2w})\in C_1$ if and only if $x_2+x_4$ equals $0$ or $2$, so the partition is equivalent to Construction \ref{C:2w_easy}. 
 
\end{enumerate}   
\end{proof} 
 
Here we will define some notions and statements that will be useful in our further arguments. 
 
\begin{lemma}\label{L:bc_differences}
Let $C=(C_1,C_2)$ be an equitable partition of $J(n,w)$ with the quotient matrix $[[a,b],[c,d]]$.
Then $$\frac{bc}{b+c}{n \choose w}=\sum_{i,i|1\leq i\leq j\leq n}{|S(f_{i,j})|},$$ where $S(f_{i,j})$ is a support for $f_{i,j}$.
\end{lemma} 
\begin{proof}
Left side of the equality is just a number of edges connecting vertices from different cells. Since every edge of the graph appears exactly once in the sum from the right side we have the equality. 
\end{proof}

Let $f\in J(n,w)\rightarrow \mathbb{R}$ then by Lemma \ref{L:zero_part_diff} the set of coordinate positions $\{1,2,\dots, n\}$ is partitioned into blocks.

Let us denote by $BD(f)$ the set of these blocks. In other words, $\forall B\in BD(f)\, \forall i,j\in B$ such that $i\neq j$ we have $f_{i,j}\equiv 0$, and  $\forall B,B'\in BD(f)$ such that $B\neq B'$ we have that $\forall i\in B\, \forall j\in B'$ $f_{i,j}\not\equiv 0$. Let $SBD(f)$ be a multiset $\{|B|:B\in BD(f)\}$. Clearly, $\sum_{a\in SBD(f)}{a}=n$ for any function $f$.

We will use the following classical result on the maximum of sum of squares of real numbers.
\begin{lemma}\label{L:squares}
Let $k\in \mathbb{N}$, $x_1,x_2,\dots , x_k,N,s \in \mathbb{R}$ such that $\forall i\in\{1,2,\dots, k\}$ we have $0\leq x_i\leq s<N$, $N\geq s\geq 0$, $ks\geq N$, and $x_1+x_2+\dots x_k=N$. Then $$\sum_{i\in \{1,2,\dots,k\}}{x_i^2}\leq s^2\lfloor\frac{N}{s}\rfloor+(N-s\lfloor \frac{N}{s}\rfloor)^2.$$ Or equivalently, $$\sum_{i,j\in \{1,2,\dots,k\}|i<j}{x_ix_j}\geq \frac{1}{2}s\lfloor\frac{N}{s}\rfloor(2N-s-s\lfloor \frac{N}{s}\rfloor).$$
\end{lemma}

As one may note, the partitions from Constructions we discussed in some sense depend only on not more than $5$ coordinate position. In other words, if we know entries of a vector in these positions, then we know exactly what cell contains the vector.  
The goal of the next Proposition is to classify all such partitions. 
\begin{proposition}\label{P:5positions}
 Let $C=(C_1,C_2)$ be equitable $2$-partitions in a Johnson graph $J(2w,w)$, $w\geq 5$, with the quotient matrix $[[w^2-b,b],[4w-2-b,w^2-4w+2+b]]$, $b\in \{2w-1,2w,\dots 4w-3\}$, $b+c=4w-2$, and $g=\frac{b\chi_{C_1}-c\chi_{C_2}}{b+c}$. Suppose that $\exists T\in BD(g)$ such that $|T|\geq 2w-5$. Then $C$ is equivalent to one of partitions from Constructions \ref{C:w}, \ref{C:2w_easy}, \ref{C:2w_hard} or \ref{C:3w}.
 \end{proposition}
\begin{proof}

Without loss of generality we consider the case $6,7,\dots, 2w\in T$ (we do not exclude a case when some of five remaining coordinates belong to $T$). Now we will try to reconstruct the whole partition $(C_1,C_2)$ using this knowledge. Let us denote by $A_{(i_1i_2i_3i_4i_5)}=\{(y_1,y_2,\dots, y_{2w})\in J(2w,w)| y_1=i_1,\dots, y_5=i_5\}$.
Clearly, $f$ does not depend on last $2w-5$ coordinates. In other words, for any fixed $i_1,i_2,\dots, i_5$ we have $A_{(i_1,i_2,i_3,i_4,i_5)}\subseteq C_1$ or  $A_{(i_1,i_2,i_3,i_4,i_5)}\subseteq C_2$.  Consider some vertex $x\in A_{(00000)}$ (we will omit comas in vectors if it does not create any misunderstandings).

 Obviously, this vertex has $5$ groups of size $w$ of neighbours and each of these groups is a subset of $C_1$ or $C_2$. Consequently, one of parameters $b$, $c$ is divisible by $w$. Therefore, there are only three putative quotient matrices: $[[w^2-3w+2,3w-2],[w, w^2-w]]$, $[[w^2-2w,2w],[2w-2, w^2-2w+2]]$, $[[w^2-3w,3w],[w-2, w^2-w+2]]$.

\begin{enumerate} 
\item $[[w^2-3w+2,3w-2],[w, w^2-w]]$

 For $w>2$ $3w-2$ is not divisible by $w$ so $A_{(00000)}\in C_2$. Without loss of generality $A_{(10000)}\subseteq C_1$ and $A_{(01000)},A_{(00100)},A_{(00010)},A_{(00001)}\subseteq C_2$. A vertex from $A_{(10000)}$ has exactly $4+(w-4)$ neighbours from $C_2$, so without loss of generality we can put $A_{(11000)}, A_{(10100)}\subseteq C_2$. Hence, all vertices from $A_{(01000)}$ and $A_{(00100)}$ have exactly $w$ neighbours from $C_1$, but vertices from $A_{(00010)}$ and $A_{(00001)}$ have only $1$ neighbour from $C_1$. The only way not to get a contradiction is to put $A_{(00011)}\subseteq C_1$. 

By Lemma \ref{L:antipodality} $A_{(i_1i_2i_3i_4i_5)}$ and $A_{((1-i_1)(1-i_2)(1-i_3)(1-i_4)(1-i_5))}$ belong to the same cell of the partition.  
So, we realized that $$C_1=\cup_{(i_1i_2i_3i_4i_5)\in B}{A_{(i_1i_2i_3i_4i_5)}},$$ where  $B=\{(10000),(11000),(10100),(00011), (01111),(00111),(01011),(11100)\}$ and $C_2=J(n,w)\setminus C_1$.            
By the arguments we checked that partition is equitable for $5$-tuples containing $0$ and $1$ ones. Lemma \ref{L:antipodality} guarantees us that we have not to check it for $5$ and $4$ ones too. Without loss of generality the cases we should check are $(11000)$,$(10010)$,$(01100)$,$(00011)$, $(01010)$. By straightforward computing one may show that vertices from corresponding sets $A$ have right number of neighbours in $C_1$ and $C_2$. 
It is easy to ensure that the partition we build is exactly the partition from the Construction \ref{C:w}.

\item $[[w^2-2w,2w],[2w-2, w^2-2w+2]]$

For $w>3$ the element $2w$ is divisible by $w$ unlike $2w-2$. Consequently, $A_{(00000)}\in C_1$.  Without loss of generality we can put $A_{(10000)},A_{(01000)}\subseteq C_2$ and $A_{(00100)},A_{(00010)},A_{(00001)}\subseteq C_1$. Our next step is to understand for what vectors from $J(5,2)$ corresponding sets $A$ must be subsets of $C_1$ and $C_2$. Vertices from $A_{(10000)}$ and $A_{(01000)}$ have exactly $3+w-4$ from $C_1$, so each of vectors $(10000)$ and $(01000)$ must be "covered" exactly one time by some vector $v$ from $J(5,2)$ such that $A_v\in C_1$. By similar arguments, every vertex from $A_{(00100)},A_{(00010)},A_{(00001)}\subseteq C_1$ has $2$ neighbours from $C_2$, so each of these three vectors must be covered twice. It is easy to see, that without loss of generality there are two non-equivalent variants to choose vectors from $J(5,2)$ in order to fulfil these requirements.
\begin{enumerate}
\item $A_{(11000)}, A_{(00110)}, A_{(00101)}, A_{(00011)}\subseteq C_1$.
All sets corresponding to remaining vectors containing two $1$s are subsets of $C_2$. Again, Lemma \ref{L:antipodality} allows to reconstruct $C_1$ and $C_2$ immediately. Let $B=\{(00000),(00100),(00010),(00001), (11000),(00110),(00101),(00011), \\ (11111),(11011),(11101),(11110), (00111),(11001),(11010),(11100) \}$, then $$C_1=\cup_{(i_1i_2i_3i_4i_5)\in B}{A_{(i_1i_2i_3i_4i_5)}}$$ and $C_2=J(n,w)\setminus C_1$.
It is easy to see that $(x_1,x_2,x_3,x_4,x_5,\dots,x_{2w})\in C_1$ if and only if $x_1+x_2$ equals $0$ or $2$, so this construction is equivalent to Construction \ref{C:2w_easy}.

\item $A_{(10100)}, A_{(01010)}, A_{(00101)}, A_{(00011)}\subseteq C_1$.
By similar arguments, we have $B=\{(00000),(00100),(00010),(00001), (10100),(01010),(00101),(00011), \\ (11111),(11011),(11101),(11110), (01011),(10101),(11010),(11100) \}$, then $$C_1=\cup_{(i_1i_2i_3i_4i_5)\in B}{A_{(i_1i_2i_3i_4i_5)}}$$ and $C_2=J(n,w)\setminus C_1$.
One may see that this construction is equivalent to Construction \ref{C:2w_hard}.
\end{enumerate}   

\item $[[w^2-3w,3w],[w-2, w^2-w+2]]$

By the same arguments as in previous item we immediately have $A_{(00000)}\in C_1$. Without loss of generality we take $A_{(10000)},A_{(01000)}\subseteq C_1$ and $A_{(00100)},A_{(00010)},A_{(00001)}\subseteq C_2$. Our next step is to understand for what vectors from $J(5,2)$ corresponding sets $A$ must be subsets of $C_1$ and $C_2$.
Every vertex from $A_{(00100)},A_{(00010)},A_{(00001)}$ has $2+(w-4)$ neighbours in $C_1$, so we conclude that $A_{(x_1,x_2,x_3,x_4,x_5)}\subseteq C_2$ if $x_1+x_2+x_3+x_4+x_5=2$ and $x_3+x_4+x_5\geq 1$. Hence, vertices from  $A_{(10000)},A_{(01000)}$ have exactly $3+3(w-1)$ vertices from $C_2$, so  $A_{(11000)}\subseteq C_1$. Lemma \ref{L:antipodality} allows to state that $$C_1=\cup_{(i_1i_2i_3i_4i_5)\in B}{A_{(i_1i_2i_3i_4i_5)}}$$ and $C_2=J(n,w)\setminus C_1$, where $B=\{(00000),(10000),(01000),(11000), (11111),(01111),(10111),(00111)\}$. Clearly, $(x_1,x_2,x_3,x_4,x_5,\dots,x_{2w})\in C_1$ if and only if $x_3+x_4+x_5$ equals $0$ or $3$. Consequently, the partition is equivalent to Construction \ref{C:3w}.

\end{enumerate}

\end{proof}

Proposition \ref{P:5positions} allows us to classify equitable $2$-partitions if there is block in $BD(g)$ of size at least $2w-5$ for corresponding function $g$. The last question we need to answer is whether there are partitions without such a block.

 \begin{theorem}\label{Partitions:n=2w}
     Let $C=(C_1,C_2)$ be an equitable partition of $J(2w,w)$ with the second eigenvalue, $w\geq 7$. Then $C$ is equivalent to one of the partitions from Constructions \ref{C:w}, \ref{C:2w_easy}, \ref{C:2w_hard} and \ref{C:3w}. For $w=4$, $w=5$ and $w=6$ the set of admissible matrices is also covered by matrices from Constructions \ref{C:w}, \ref{C:2w_easy}, \ref{C:2w_hard} and \ref{C:3w}.  
 \end{theorem}
 
 \begin{proof}
   Let $C$ have a quotient matrix $[[w^2-b,b],[4w-2-b,w^2-4w+b+2]]$ for some $b\in \{2w-1,2w,\dots 4w-3\}$. Again let $f=b\chi_{C_1}-c\chi_{C_2}$ and $g=\frac{b\chi_{C_1}-c\chi_{C_2}}{b+c}$. Clearly, $g$ is a $\lambda_1(2w-2,w-1)$-eigenfunction if $J(2w-2,w-1)$.
     
   In Proposition \ref{Derivatves:n=2w} we described all partitions having one of partial differences of $g$ being equivalent to $f_3$ from Theorem \ref{T:Derivates}. Therefore, we conclude that the set $PD(g)=\{g_{i,j}|1\leq i< j \leq n\}$ consists of functions which are equivalent to the all-zero function, $f_1$ or $f_2$ from Theorem \ref{T:Derivates}. Let us denote by $k_0$, $k_1$ and $k_2$ respectively the numbers of such functions in $PD(g)$. Clearly, the size of the support of them is equal to $0$, $2{2w-4 \choose w-2}$ and ${2w-4 \choose w-3}$ respectively. Therefore, by Lemma \ref{L:bc_differences} we have the following system of equations:
        
 $$\begin{cases}
k_0+k_1+k_2={2w \choose 2}\\
\frac{bc}{b+c}{2w \choose w}=k_1{ 2w-4 \choose w-2}+k_2{ 2w-4 \choose w-3}.\end{cases}    $$    
After simplifying the second equation and using $b+c=4w-2$ we get          
        \begin{equation}\label{E:k1+k2}
         \begin{cases}
k_0+k_1+k_2={2w \choose 2}\\
bc(2w-3)=k_1w(w-1)+k_2w(w-2).\end{cases}    
\end{equation}         
 The rest of the proof is based on the analysis of $BD(g)$ and system (\ref{E:k1+k2}).
 Since $(\frac{b+c}{2})^2\geq bc$, (\ref{E:k1+k2}) gives us $(2w-1)^2(2w-3)\geq k_1w(w-1)+k_2w(w-2)\geq w(w-2)(k_1+k_2).$ After simplifying we have that the number of non-zero partial differences of $g$ is not greater than $8w-2$ for $w\geq 4$, $(k_1+k_2)\leq 8w-2$.
 
 Now let us consider the case $w\geq 9$. Suppose that $SBD(g)$ contains an element $t$ such that $6\leq t\leq 2w-6$. Consequently, $k_1+k_2\geq t(2w-t)\geq 6(2w-6)=12w-36$. So $8w-2\geq 12w-36$, but it not possible for $w\geq 9$. Now suppose, that $\forall t\in SBD(g)$ we have $t<6$. By Lemma \ref{L:squares} we have $$8w-2\geq \frac{1}{2}5\lfloor\frac{2w}{5}\rfloor(4w-5-5\lfloor\frac{2w}{5}\rfloor)\geq \frac{(2w-4)(2w-5)}{2}.$$ This inequality leads us to $4w^2-34w+24\leq 0$ which is not true for $w\geq 9$.
 
Therefore, we conclude that $\exists T\in BD(g)$ such that $|T|\geq 2w-5$. Here we can use Proposition \ref{P:5positions} and claim that $C$ is equivalent to one of partitions from Constructions \ref{C:w},\ref{C:2w_easy},\ref{C:2w_hard},\ref{C:3w} with $b=3w-2$, $b=2w$, $b=2w$ and $b=3w$ respectively.
 Now we are going to consider cases $w=8,7$. 
 \begin{enumerate}
 \item $w=8$.
 From (\ref{E:k1+k2}) we have $13bc=56k_1+48k_2$, where $b+c=30$ and $b\geq c$. If $SBD(g)$ contains an element $t\geq 2w-5=11$ then by Proposition \ref{P:5positions} and by the arguments we provided above $C$ is equivalent to one of partitions from Constructions \ref{C:w},\ref{C:2w_easy},\ref{C:2w_hard},\ref{C:3w}. Otherwise, $\forall t\in SBD(g)$ we have $t\leq 10$ and Lemma \ref{L:squares} gives us $k_1+k_2\geq 60$. Therefore, $13bc\geq 48*60$ together with $b+c=30$ and $8|bc$, we have $b=16$ and $c=14$. Then we have $7k_1+6k_2=364$. Since $k_1$ and $k_2$ are nonnegative integers and $k_1+k_2\geq 60$, we conclude that $k_1+k_2=60$. In other words, $SBD(g)=\{10,6\}$. Without loss of generality let first six coordinate positions form this block of size $6$. Consider a vertex of $J(16,8)$ with zeros in these coordinate positions. Obviously this vertex may has only $0$ or $48$ vertices from other cell and we get a contradiction.
  \item $w=7$. From (\ref{E:k1+k2}) we have $11bc=42k_1+35k_2$, where $b+c=26$ and $b\geq c$. Again we only consider the case when there are no blocks of size at least $9$ in $BD(g)$. Therefore, by Lemma \ref{L:squares} for $N=14$, and $s=8$, we have $k_1+k_2\geq 48$. Consequently, simplifying $11bc\geq 35*48$ we have $bc\geq 153$. Together with the restrictions on $b$ and $c$ it gives us $b=14$, $c=12$ and $264=6k_1+5k_2$. 
  
  Let $BD(g)$ contain an element of size $8$. If $SBD(g)\neq \{8,6\}$ then all elements of $SBD(g)\setminus\{8\}$ are not greater than $5$. Therefore, by Lemma \ref{L:squares} $k_1+k_2\geq 8*6+5=53$ and we get a contradiction with $264=6k_1+5k_2$. So we conclude, that $SBD(g)=\{8,6\}$.  Without loss of generality let first six coordinate positions form this block of size $6$. Consider a vertex of $J(14,7)$ with zeros in these coordinate positions. Obviously, this vertex may has only $0$ or $42$ vertices from other cell and we get a contradiction.      

  Let $BD(g)$ contain an element of size $7$. If $SBD(g)\neq \{7,7\}$ then all elements of $SBD(g)\setminus\{7\}$ are not greater the $6$. Therefore, by Lemma \ref{L:squares} $k_1+k_2\geq 7*7+6=55$ and we get a contradiction with $264=6k_1+5k_2$.  So we conclude, that $SBD(g)=\{7,7\}$. Continuing in the same manner we see that a vertex having zeros in positions from one of these blocks can have only $0$ or $49$ vertices from other cell.  
  
 If all blocks in $BD(g)$ have size not greater than $6$ then by Lemma \ref{L:squares} for $N=14$, and $s=6$, we have $k_1+k_2\geq 60$ that contradicts to $264=6k_1+5k_2$.

 \end{enumerate}
 For $w\in \{4,5,6\}$, by system of equations (\ref{E:k1+k2}) we have that $bc$ is divisible by $w$. Using the equality $b+c=4w-2$, it easy to find all possible pairs $b$, $c$ and ensure that each of these possible pairs belongs to one of Constructions \ref{C:w}, \ref{C:2w_easy}, \ref{C:2w_hard} and \ref{C:3w}.      
 
 \end{proof}

\section{Acknowledgements}

The author is grateful to Ivan Mogilnykh for interesting discussions.

\end{document}